\documentclass[12pt,reqno,letterpaper]{amsart}
\usepackage{epsfig,amsmath,amssymb,amscd,color,url}
\usepackage[percent]{overpic}
\usepackage{enumitem}
\setlist{itemsep=1pt,topsep=-5pt,parsep=0pt,partopsep=-2pt}

\newtheorem{thm}{Theorem}[section]
\newtheorem{cor}[thm]{Corollary}
\newtheorem{lem}[thm]{Lemma}
\newtheorem{pro}[thm]{Proposition}

\theoremstyle{definition}

\newtheorem{remark}{Remark}

\numberwithin{equation}{section}

\def\bq{\mathbf{q}}

\def\cA{\mathcal{A}}

\def\cM{\mathcal{M}}
\def\cQ{\mathcal{Q}}
\def\cR{\mathcal{R}}
\def\cS{\mathcal{S}}
\def\cU{\mathcal{U}}

\def\hf{\hat{f}}
\def\hh{\hat{h}}
\def\hp{\hat{p}}
\def\hq{\hat{q}}
\def\hx{\hat{x}}

\def \hU{\widehat{U}}

\begin{document}

\title[Homoclinic points for convex billiards]{Homoclinic points for convex billiards}

\author[Z. Xia]{Zhihong Xia}
\address{Department of Mathematics,
Northwestern University,
Evanston, IL 60208.}
\email{xia@math.northwestern.edu}

\author[P. Zhang]{Pengfei Zhang}
\address{Department of Mathematics and Statistics,
University of Massachusetts Amherst,
Amherst, MA 01003.}
\email{pzhang@math.umass.edu}



\begin{abstract}
In this paper we investigate some generic properties of a billiard system on a convex table.
We show that generically, 
every hyperbolic periodic point admits some homoclinic orbit.
\end{abstract}

\maketitle

\section{Introduction}

Let $Q\subset \mathbb{R}^2$ be a strictly convex domain,
whose boundary $\partial Q$ is $C^r$ (for some $r\ge 2$) 
and has non-vanishing curvature everywhere. 
There is a induced dynamical system on $Q$, the {\it dynamical billiard}, in which
a particle moves freely along strict segments in $Q$, 
and changes its velocity according to the law of elastic reflection
upon collisions with the boundary $\partial Q$.
Let $M$ be the phase space of the billiard system, which consists of all post-collision statuses
with the boundary $\partial Q$, and  $F:M\to M$ be induced {\it billiard map},
which is a $C^{r-1}$ diffeomorphism on $M$. Since the motion is free inside $Q$,
the dynamics of the billiard systems are determined completely by the shapes 
of the table $Q$. For example, the system on an elliptic table is completely integrable
(see \cite{Bir}).
In this paper we consider a generic convex billiard table, 
and prove the existence of homoclinic orbits on all branches of
all hyperbolic periodic points.

There are many results on the existence of homoclinic or heteroclinic orbits. 
Poincar\'e first appreciated the significance of transversal homoclinic point
(developed later by Birkhoff and Smale).
Robinson proved in \cite{Rob73} that for a $C^r$ map on two-sphere,
if the unstable manifold of a hyperbolic fixed point accumulates on its stable manifold,
then a $C^r$ small perturbation can create a homoclinic point. Pixton \cite{Pix} extended
Robinson's result to periodic orbits, 
and proved the generic existence of homoclinic orbit on $S^2$ in the area-preserving category. 
Using some topological argument,
Oliveira showed in \cite{Oli1} the generic existence of homoclinic orbit on $\mathbb{T}^2$, and 
extended in \cite{Oli2} to any compact surface (for the area-preserving systems 
with irreducible action on the first homology group $H_1(M)$).
Recently in \cite{Xia2}, Xia gave a proof of the the generic existence of homoclinic orbit 
among the class of systems homotopic to identity,
and among the class of Hamiltonian systems (also on general compact surface).

Creation of homoclinic orbits in billiard system turns out to 
be completely different from the above approaches,
since we can't perturb the billiard map directly, but deform the underlining billiard table.
Such types of perturbations are {\it not the  local ones}, and have some unavoidable side effects.
Using a variational approach, Bangert showed in \cite{Ban} that for {\it all} monotone  twist maps, 
for each rotation type $\frac{q}{p}$, the set $\cM_{q/p}$ of global minimizers
are totally ordered, and there exist heteroclinic orbits for all {\it neighboring pairs} of periodic points
(and homoclinic orbits if there is only one periodic point).
Since the billiard map on convex table is a monotone twist map, 
there always exist some heteroclinic orbits.
However, most periodic orbits are not global minimizers in general, and there
may not be any heteroclinic orbits relating them.
It is proved in \cite{DOP2} that for generic billiard table $Q$, if $W^u(p)\cap W^s(q)\neq\emptyset$,
then there will be a transversal intersection. See also \cite{LeTa,Don}.
A priori, the two manifolds may have no intersection at all.
What we show here is that, non-intersection is not the generic case if $q=p$,
and every hyperbolic periodic points does admit some homoclinic orbits.
\begin{thm}
Let $r\ge3$. For a $C^r$-generic convex billiard system $Q$, we have that
\begin{enumerate}
\item all periodic points are either hyperbolic or elliptic \cite{DOP2};

\item for every hyperbolic point $p$, there is a transversal homoclinic point
$x\in W^u(p)\pitchfork W^s(p)$.
\end{enumerate}
\end{thm}
Note that the properties we listed are open for all finite periods, and $G_\delta$-properties in general. 
So it suffices to prove the theorem in the $C^\infty$-category.
Then $C^r$-genericity  follows directly.

It is worth pointing out that the existence of heteroclinic orbits is not generic,
since two hyperbolic orbits may be separated by some invariant curves.
Specifically, let's consider the elliptic table $x^2+2y^2\le 1$, on which the billiard system
is completely integrable, and  the periodic two orbit along the minor axis is parabolic.
A $C^\infty$-small deformation of $Q$ can turn this orbit into a hyperbolic periodic orbit.
But there is no heteroclinic orbit connecting the periodic orbits on major and minor axes, 
since some of the invariant curves survive and separate the two orbits.

Let's consider a general diffeomorphism $f:M\to M$ on a closed surface $M$, 
$p$ and $q$ be two hyperbolic 
periodic points of $f$.
Recall that $f$ is said to have a {\it homoclinic loop} at $p$,  
if one unstable branch of $p$ coincides with a stable branch of $p$.
Similar, we say that $f$ has a {\it heteroclinic path} from $p$ to $q$,
if one unstable branch of $p$ coincides with a stable branch of $q$.
In general we may say that the common branch forms a {\it path of saddle connection} of $f$.
A main step proving above theorem is
\begin{pro}
Let $f$ be an area-preserving diffeomorphism on a closed surface $M$, such 
 that each periodic point is either hyperbolic, or elliptic and Moser stable.
Then for any branch $L$ of the invariant manifolds of any hyperbolic periodic point $p$, 
one of the following alternative holds:
\begin{enumerate}
\item $L$ contains some recurrent point.
\item $L$ forms a path of saddle connections. 
\end{enumerate}
\end{pro}
The existence of a path of saddle connection is highly non-generic.
But the point of above proposition is not about generic systems.
It is for all systems satisfying above condition. 
In particular, for a specific dynamical system, if one can pre-exclude
any path of saddle connections,  then all stable and unstable branches of {\it the given system}
must be recurrent (no further perturbation required). 
Then homoclinic intersection follows if the phase space
$M$ has simple topology.

\section{Preliminary}

In this section we describe the topology on the collection of convex billiard tables,
list some generic properties of the billiard maps on these tables, 
and then introduce the main tool for our study, prime-end extensions. 

\subsection{Topology on the set of tables}
Let $Q\subset \mathbb{R}^2$ be a strictly convex domains,
whose boundary $\partial Q$ is a $C^\infty$ curve with nonvanishing curvature.  
The support function $h$ of $Q$ is given by
$h(\theta)=\sup\{\mathbf{q}\cdot \langle\cos\theta,\sin\theta\rangle:\mathbf{q}\in Q\}$. The collection of all such support functions,
say $\cS^\infty$, form an open subset of $C^\infty(S^1,\mathbb{R})$ 
(with the induced topology on $\cS^\infty$).
We denote by $\cQ^\infty$  the corresponding collection\footnote{Sometime 
we consider the quotient space of $\cQ^\infty$ by rigid motions on $\mathbb{R}^2$
and by homotheties of $Q$, both of which do not change the dynamics at all.
The resulting space is also a Baire space (see  \cite[Prop. 3]{DOP2}), 
and shares the same generic properties.} 
 of convex billiard tables.
A property (P) is said to be true on a $C^\infty$-generic convex billiard system,
if the set of the support functions of the tables satisfying (P) contains a $C^\infty$-residual
subset of $\cS^\infty$.

Let $M=\partial Q\times(-1,1)$ be the phase space of the billiard map on $Q$,
where $\partial Q$ is parametrized by the arc-length parameter $r=r(\bq)$,
$s=\sin\varphi$ and $\varphi$ measures the angle of the outgoing velocity with the inner normal
direction at $\bq\in\partial Q$. The advantage of choosing this coordinate system is that 
the billiard map always preserves the Lebesgue measure $m=drds$.
It is well known that for convex billiard maps,
there is no fixed point in the open annulus $-1<s<1$,
and period-$2$ orbits can only lie in the line $s=0$.

\subsection{Limit sets of invariant manifolds}

Let $M$ be a compact surface, $f:M\to M$ be an area-preserving diffeomorphism on $M$
and $p$ be a hyperbolic periodic point of $f$ with period $n$, 
 $L$ be a branch of the unstable manifold of $p$. Pick a point $x\in L$.
Then $L[x,f^{2n}x]$ forms a fundamental domain of $L$, and its omega-limit set under $f^{2n}$
is denoted by $\omega(L,x,f^{2n})$. Clearly it is independent the choice of $x\in L$
and remains the same if we use $f^{2nl}$. So we may denote it by $\omega(L)$ for short.
Similarly we define the alpha-limit set $\alpha(L)$ of a stable branch $L$
of $p$ with respect to $f^{-2n}$.
Sometime it is more convenient to use a unified notation. That is,
we may use $\omega(L)$ to denote the {\it alpha-limit set} of $L$, if $L$ is a stable branch.
\begin{pro}[\cite{Oli1}]
Let $L,K$ be two branches of some hyperbolic periodic points of $f:M\to M$. Then 
\begin{itemize}
\item either $L\subset \omega(L)$, or $L\cap\omega(L)=\emptyset$.

\item either $K\subset \omega(L)$, or $K\cap\omega(L)=\emptyset$.
 \end{itemize}  
\end{pro}  
We will provide a slightly more general statement in Lemma \ref{branch}
with a much shorter proof.

\subsection{Some known properties}\label{Un}

Let $n\ge 1$ and $\cU_n$ be the set
of convex tables in $\cQ^\infty$ such that for every $Q\in \cU_n$,
\begin{enumerate}
\item[(a)] $\mathrm{Per}_n(Q)=\{x\in M: F^nx=x\}$ is finite,
and all of them are either hyperbolic or elliptic;

\item[(b)] given two hyperbolic periodic points $p,q\in\mathrm{Per}_n(Q)$, 
$W^u(p)$ and $W^s(q)$ either don't intersect, or have some transversal intersections;

\item[(c)]  every elliptic periodic point in $\mathrm{Per}_n(Q)$ is Moser stable. 
\end{enumerate} 
Recall that an elliptic fixed point $p$ is said to be {\it Moser stable}, 
if there is a fundamental system of neighborhoods which are invariant closed disks $D$
surrounding the point, such that $F|_{\partial D}$ is a minimal rotation. 
So Moser stable periodic points are isolated from the dynamics
and are unreachable from the outside invariant rays.

\begin{remark}
The above Property (b) is chosen carefully when defining $\mathcal{U}_n$, 
since there might be some tangent intersections of two invariant manifolds
besides the transversal ones (even in the generic case, as pointed out in \cite{DOP2}).
This is different from the classical {\it Kupka-Smale condition}
of generic diffeomorphisms, which requires that all intersections of two invariant manifolds
are transversal.
This might be a technique difficulty, since only special perturbations can be performed:
deformations of the underlining table.
\end{remark}

\begin{pro}\label{finite}
For every $n\ge1$, the set $\cU_n$ is $C^2$-open and $C^\infty$-dense in $\cQ^\infty$.
\end{pro}
This proposition is the combination of Propositions 6--9 in \cite{DOP2} and Theorems 3--4 in \cite{BuGr}.
In fact they made several interesting observations:
\begin{enumerate}
\item to perturb a degenerate periodic orbit,  a finite induction process is employed 
to find a local deformation, {\it the normal perturbations} of $Q$, 
that changes only the curvature at some base point
of the periodic orbit, while keeping the orbit unchanged. 

\item every $k$-periodic orbit $\{F^ix:0\le i<k\}$ of $F$ gives an inscribed $k$-polygon in $Q$. 
So at least one of the $F^ix$ lies in the compact annulus 
$\cA_k=\{(r,s):|s|\le \sin(1/2-1/k)\pi\}$ (recall $s=\sin\varphi$ in our notation).
Therefore the finiteness follows from non-degeneracy of periodic orbits. 

\item to make an elliptic periodic orbit Moser stable, 
it suffices to make a $C^\infty$-small normal deformation maintaining a third order contact, such
that the first Birkhoff coefficient is nonzero under the Birkhoff normal form.
Then Moser twist mapping theorem ensures that the orbit is Moser stable. 
\end{enumerate}
See also \cite{LeTa,Don,DOP1} for earlier versions of such deformations.

Another classical result we may use is
\vskip.2cm

\noindent{\bf Lazutkin Theorem.} \cite{Laz}
{\it  For $r$ large enough, every strictly convex domain $Q$ with $C^r$ boundary 
and nowhere vanishing curvature admits 
a family of invariant curves converging to the boundary $\partial M=\{s=\pm 1\}$.
Moreover, these invariant curves have irrational rotation numbers and occupy a positive 
volume in a neighborhood of $\partial M$.
 }

Originally this theorem asked for $r\ge 553$ to apply KAM theory; 
later this was reduced to $r=6$, see \cite{Dou}. 
\begin{remark}
Clearly the existence of invariant curves prevents the billiard system to be ergodic.
So all convex billiards are not ergodic if the boundary is $C^6$ and with nowhere vanishing curvature. 
On the other hand, it is showed that no such invariant curve can survive if one of the restrictions is relaxed,
see Mather \cite{Mat} and Hubacher \cite{Hub}.
Currently, most known ergodic convex billiard tables have long flat components to
apply defocusing mechanism (see \cite{Bu,Woj,BuGr} and the references therein).  
\end{remark}

\subsection{Prime-end compactification and extension}

Let $\mathbb{D}\subset\mathbb{R}^2=\mathbb{C}$ be the open unit disk,
$U$ be a bounded, simply connected domain on the plane.
Then there exists a conformal map $h:\mathbb{D}\to U$ (see \cite{Mat1}). 
Every point $\hx\in S^1=\partial \mathbb{D}$ is specified by a nested sequence of open arcs
$\gamma_n\subset U$ with $|\gamma_n|<1/n$ such that the endpoints of $h^{-1}\gamma_n$
lie on each side of $\hx\in S^1$, $h^{-1}\gamma_n$ are nested and converge to $\hx$ 
(see \cite[Theorem 17.11]{Mil}).
The prime-end compactification $\hU\triangleq U\sqcup S^1$ gives one way to compactify $U$,
whose topology is uniquely determined by 
the extended homeomorphism  $\hh:\overline{\mathbb{D}}\to\hU$,
such that $\hh|_{\mathbb{D}}=h$ and $\hh|_{S^1}=\text{Id}$ (see \cite[Theorem 17.12]{Mil}).
Here we point out that:

\begin{enumerate}
\item[(I.1)] A point $\hx\in S^1$ may cover many points in the boundary $\partial U$, 
which is compact and connected, and called the {\it impression} of $\hx$ 
\cite[Theorem 17.7]{Mil}.

\item[(I.2)] A point $x\in \partial U$ may be lifted to many points in the prime-end $S^1$,
because we may observe it several times when traveling inside $U$ (\cite[Figure 37-(b)]{Mil}).
\end{enumerate}
See also \cite[\S 7]{Mat1} for various examples.

Let $f:U\to U$ be a homeomorphism. Then there exists uniquely an 
extension $\hf:\hU\to \hU$. If $f$ is orientation-preserving,
then the restriction $\hf|_{S^1}$ is an orientation-preserving circle homeomorphism.
Therefore the rotation number $\rho$ of $\hf|_{S^1}$ is well defined, which is 
called the {\it Carath\'eodory rotation number} associated to $(f,U)$, see \cite{Mat2}. 
It is well known that $\hf|_{S^1}$ has periodic orbits if and only if it has a rational rotation number.
However, we note that
\begin{enumerate}
\item[(II.1)] A fixed point of $\hf$ on $S^1$ doesn't imply the existence of a fixed point
of $f$ on $\partial U$. For example, let 
$U$ be the unstable set of the new formed source of a DA-map on $\mathbb{T}^2$,
which is an open set bounded by two (split) stable manifolds. Then we can construct $f$ 
on $U$ which are translations along the boundary. So no point on $\partial U$ is fixed by $f$. 
But the two added points in the prime-end are fixed by $\hf$.

\item[(II.2)] A fixed point of $f$ on $\partial U$ doesn't imply the existence of a fixed point
of $\hf$ on $S^1$. For example, we can put an aster set $E$ on the north pole $N$,
and let $f$ rotate $S^2$
along the axis by $\pi/3$. Although $N$ is fixed, the prime-end extension on
$U=S^2\backslash E$ has rotation number $1/6$ on the added boundary.
\end{enumerate}

Prime-end extensions can also be defined for connected closed surface $M$
(possibly with boundary).
Let $U\subset M$ be an open connected subset,
whose boundary consists of a finite number of connected pieces,
each of which has more than one point. Then we can add to $U$ a finite number of circles,
to get its prime-end compactification $\hU$. See also \cite{Mat1,Xia1}.
\begin{lem}\label{downfixed}
Let $f:M\to M$ be an area-preserving diffeomorphism,  
$U$ be an open connected subset and $\hf:\hU\to \hU$ be a prime-end extension.
\begin{enumerate}
\item If there exists a fixed point of $\hf$ on the added boundary of $\hU$,
then there exists a fixed point of $f$ on $\partial U$.

\item Moreover assume that each fixed point of $f$ is either hyperbolic, or elliptic and Moser stable.
Then the fixed point is hyperbolic.
\end{enumerate}
\end{lem}
\begin{proof}
(1). Let $\{\gamma_n\}$ be a collection of open arcs with $|\gamma_n|\le 1/n$ 
and both endpoints on $\partial U$, that define the point
$\hp$. Note that $\{f\circ\gamma_n\}$ also defines $\hp=\hf\hp$.
Let $p\in \partial U$ be an accumulating point of $\gamma_n$.
Since $f$ is area-preserving, we see that the fundamental domains separated by
$\gamma_n$ and $f\gamma_n$ have the same area and intersect each other.
So $f\gamma_n\cap\gamma_n\neq\emptyset$.
Therefore, $p$ must be fixed by $f$ (since $|\gamma_n|\le 1/n$).
 
(2). If $p$ is elliptic, then it is also Moser stable (by our assumption).
So there are a fundamental system of closed disks $D$ surrounding $p$ such  that $f|_{\partial D}$
is minimal. For any such a disk $D$, if $\partial D\cap U\neq\emptyset$, then
$\partial D\subset U$ (by the invariance of $U$). So either $\partial D\cap U=\emptyset$
for some $D$, which implies that $\gamma_n$ are stopped outside $D$ and can't accumulate $p$ 
(a contradiction); or  $\partial D\subset U$ for all such $D$'s.
Since $\partial U$ is assumed to have only finite number of connected pieces, 
there exists some disk $D$ with  $D\subset U$, contradicting the choices of $\gamma_n$. 
So $p$ must be hyperbolic.
\end{proof}

\section{Recurrence and Saddle connections}

In this section we will study the behavior of the branches of hyperbolic periodic points
of a surface diffeomorphism. 
More precisely, let $M$ be a closed connected surface of genus $g\ge0$,
$f: M\to M$ be an area-preserving diffeomorphism on $M$.
\begin{lem}\label{branch}
Let $K\subset M$ be a compact, connected, $f$-invariant subset, $\gamma$ be a branch
of the invariant manifolds of a hyperbolic fixed point $p$. Then either $\gamma\subset K$,
or $\gamma \cap K=\emptyset$.
\end{lem} 
This was proved by Mather in \cite{Mat2} using the ideal boundary closure. 
A different proof was given in \cite{FrLC}
for the case $M=S^2$. 
For completeness we give a proof without using the ideal boundary points.
\begin{proof}
Without loss of generality we assume $\gamma$ is a stable branch fixed by $f$.
Suppose on the contrary that $\gamma$ intersects both $K$ and $U=M\backslash K$.
Let $\alpha$ be an open arc in $\gamma$ with both endpoints on $K$.
Clearly $p$ can't be the endpoint of $\alpha$ (otherwise $\gamma\subset U$ since $U$
is invariant). Let $g\ge0$ be the genus of $M$,
and consider the iterates $f^i\alpha$, $i=0,\cdots ,2g$, all of which have their endpoints on $K$.
Together with $K$, they will separate $M$ into (at least two) different connected components.
Pick any one of components, say $V$, that touches $\alpha$ from one side. 

Fix a point $x\in \alpha$ and pick a small $\delta>0$ such that $B(x,2\delta)\cap K=\emptyset$. 
Let $U_x=B(x,\delta)\cap V$ (close to a half disk if $\delta$ is small enough). 
Let $\epsilon<\delta/3$, and $V_\epsilon:=V\backslash B(\partial V,\epsilon)$ 
be the $\epsilon$-interior of $V$,
which may consist of several connected components. Pick a point $y\in V_\epsilon$, and let
$V_\epsilon^\ast$ be the connected component of $V_\epsilon$ containing $y$.
It is easy to see that $\bigcup_{\epsilon>0}V_\epsilon^\ast$ is closed in $V$, 
hence coincides with $V$ (since $V$ is connected). 
Pick $\epsilon$ small enough such that $m(V_\epsilon^\ast)\ge m(V)-\delta^2$.
In particular, it forces $U_x\cap V_\epsilon^\ast\neq\emptyset$.

Pick $N$ large enough such that $|f^n\alpha|<\epsilon$ for all $n\ge N$.
By Poincar\'e Recurrence Theorem,
there exists $k\ge N$ such that $f^kU_x\cap U_x\neq\emptyset$.
Note that $f^{k+i}\alpha$, $0\le i\le 2g$ are the only pieces of the boundary of $f^kV$ that can enter $V$.
With both endpoints hanging on $K$, all of them are too short 
(less than $\epsilon$) to invade the ball $B(x,\delta)$ (since $B(x,2\delta)\cap K=\emptyset$).
So $B(x,\delta)\subset f^kV$.
Similarly,  we have $V_\epsilon^\ast\subset f^{k}V$.
This implies that $m(f^kV)> m(V_\epsilon^\ast)+\pi\delta^2/3>m(V)$, 
which contradicts the area-preserving assumption.
This completes the proof.
\end{proof}

By Hartman's theorem, for any hyperbolic fixed point $p$ of $f$,
there exist a $C^0$ local chart $W$ around $p$ and some $\lambda>1$,
such that the $x$-axis gives the local unstable manifold,
the $y$-axis gives the local stable manifold, 
and $f(x,y)=(\lambda x,\lambda^{-1}y)$.
Next lemma can be proved following the ideas of the proof of \cite[Lemma 6.3]{FrLC}
(see also \cite{Mat1}).
Suppose that each fixed point of $f$ is either hyperbolic, or elliptic and Moser stable.
Let $K\subset M$ be a compact, connected, $f$-invariant subset,
$U$ be an invariant, connected component of $M\backslash K$.
\begin{lem}\label{open}
If there is some fixed point of $\hf$ on the added boundary of $\hU$,
then there exists a hyperbolic fixed point $p$ on $\partial U$ and
a small open neighborhood $W$ of $p$, such that $W\cap \partial U$ 
consists of one or two branches of the local invariant manifold of $p$.
\end{lem}
\begin{proof}
Let $\hp$ be a fixed point of $\hf$. Then by Lemma \ref{downfixed},
there exists a hyperbolic fixed point $p$ on $\partial U$. 
Let $W$ be a small neighborhood of $p$ given by Hartman's theorem.
There are two cases and we treat them separately:

\noindent {\bf Case 1.} $\gamma\cap U\neq\emptyset$ for some local branch $\gamma$ of $p$ in $W$.
In this case we have $\gamma\subset U$ (by Lemma \ref{branch}),
since $U$ is a connected component of $M\backslash K$. 
Moreover, the two open quadrants on both sides of $\gamma$
are contained in $U$ (by the invariance of $U$). Now we consider the new-reached branches, 
say $\hat\gamma$.
Either $\hat\gamma$ doesn't intersect $U$ (then it serves as a boundary piece of $U$), 
or we can continue
the above argument such that $U$ contain a new quadrant adjacent to $\hat\gamma$, 
and push the boundary to the next local branch.

Inductively, we see that either $U$ stops at some branch, 
which serves as the boundary piece of $U$ (so our conclusion holds in this case), 
or we keep going such that all four branches intersect $U$. 
In the later case, $U$ contains the whole open neighborhood of $p$, a contradiction.

\noindent {\bf Case 2.} $\gamma\cap U=\emptyset$ for all local branches $\gamma$. 
Let $\{\gamma_n\}$ be a collection of open arcs with $|\gamma_n|\le 1/n$ 
that define the point $\hp$ and locate the point $p$. 
Passing to a subsequence and relabeling the quadrants, we 
assume $\gamma_n$ lies in the first quadrant for all $n$. Then consider the curves 
$\Gamma_\epsilon=\{(x,\epsilon):0\le x\le \epsilon\}\cup \{(\epsilon,y):0\le y\le\epsilon\}$,
and the rectangles $R_\epsilon$ in the first quadrant bounded by $\Gamma_\epsilon$.
Pick $\epsilon>0$ small enough such that $R_\epsilon$ is disjoint from $\gamma_1$. 
Still $R_\epsilon$ contains $\gamma_n$ for some large enough $n=n_\epsilon$.

All $\gamma_n$ are chosen from a connected open set $U$. 
So $\Gamma_\epsilon\cap U\neq\emptyset$, and some component of this intersection separates
$\gamma_1$ and $\gamma_n$ in $U$, say $\beta_\epsilon$, 
which is an open arc in $U$, with both endpoints on $K$.

Let $U_1$ and $U_2$ be the two components of $U$ separated by $\beta_\epsilon$.
Since $f\gamma_n\cap \gamma_n\neq\emptyset$ for all $n\ge1$, we see that
$fU_i\cap U_i\neq\emptyset$ for both $i=1,2$ 
(since they contain $\gamma_1$ and $\gamma_{n}$, respectively).
Being the only boundary piece of $U_i$ in $U$, $f\beta_\epsilon$ can't lie on one side of
$\beta_\epsilon$ (by the area-preserving property of $f$).
So $f\beta_\epsilon$ must meet $\beta_\epsilon$. 
By the choice of Hartman coordinate system around $p$, 
we see that $f\Gamma_\epsilon\cap\Gamma_\epsilon$
consists of a single point $a_\epsilon=(\epsilon,\epsilon\lambda^{-1})$, 
which must lies on  $\beta_\epsilon\cap f\beta_\epsilon$. Then it is easy to see
 its pre-image $f^{-1}a_\epsilon=(\epsilon\lambda^{-1},\epsilon)$ also lies in $\beta_\epsilon$.

According to the definition of $\beta_\epsilon$,
we see the whole piece $\Gamma_\epsilon(f^{-1}a_\epsilon,a_\epsilon)$ is contained in 
$\beta_\epsilon$,
which is also contained in $U$.
Moreover, this holds for all smaller $\epsilon$. So $U$ contains a cone in the first quadrant,
and then the whole first quadrant by the invariance of $U$. In this case, the boundary
consists also the local branches. This finishes the proof.
\end{proof}

\begin{pro}\label{recurrent}
Suppose that each periodic point of $f$ is either hyperbolic, or elliptic and Moser stable.
Then for any branch $L$ of the invariant manifolds of any hyperbolic periodic point $p$, 
we have
\begin{enumerate}
\item  either $L$ contains some recurrent point.
\item or $L$ forms a saddle connection. 
\end{enumerate}
\end{pro}
\begin{proof}
Without loss of generality, we assume $p$ is fixed and all four quadrants are also fixed by $f$,
$L$ is an unstable branch, say the positive $x$-axis, which is also fixed by $f$.
Suppose that $L$ contains no recurrent point. 
Then $L$ can't approach $p$ from the first and fourth 
quadrants, and $K:=\overline{L}=\{p\}\sqcup L\sqcup \omega(L)$.
Therefore,  $M\backslash K$ contains the first and fourth quadrants of a smaller neighborhood $W$,
which will grow to two open strips attaching to $L$ (by the invariance of $M\backslash K$).
Let $U$ be the connected component of $M\backslash K$ containing the first quadrant.
It is easy to see that $U$ is invariant (being a component).
Let $\hU=U\sqcup S^1$ be the prime-end compactification of $U$, $\hf: \hU\to \hU$
be the prime-end extension and $g=\hf|_{S^1}:S^1\to S^1$ be the induced
 circle homeomorphism.

Note that each point $x\in L$ is a regular point of the boundary $\partial U$, 
and gives rise to at least one point in $\hU$,
say $\hx\in S^1\subset \hU$ (when observing $L$ from the first quadrant). 
Denote by $\hp$ the one-side limit of $\hx$ in $S^1$ as $x$ moves back to $p$ along $L$.
Clearly $\hp$ is a fixed point of $g$.  So the rotation number of $g$ is zero,
the omega-set $\omega(\hx)$ ($x\in L$) consists of a single fixed point of $g$, say $\hq$.
Note that if $L$ never comes back to $p$, then $U$ also contains the fourth quadrant 
and we may get two points when observing $L$ from the the first and fourth quadrants respectively, say
$\hq_{+}$ and $\hq_{-}$. Let's take $\hq=\hq_{+}$ in that case for certainty.
Then by Lemma \ref{downfixed}, there exists a hyperbolic fixed point $q\in K$.

\noindent{\bf Case 1.} $q=p$. According to Lemma \ref{open},
there exists a small neighborhood $W$ of $p$, such that
$W\cap \partial U$ consists only of the branch(es) of local invariant manifolds.
Each piece is regular on $\partial U$, and hence contained in $L$.
Clearly $L$ can not intersect the unstable branches. So $L$ must coincide with
one  branch of the stable manifold of $p$, which forms a homoclinic loop.

\noindent{\bf Case 2.} $q\neq p$.
Applying Lemma \ref{open} again, we see that there exists a small neighborhood $V$ of $q$,
such that $V\cap \partial U$ consists only of the branch(es) of the local invariant manifolds of $q$.
Each piece is regular on $\partial U$, and hence contained in $L$.
Once again it is a stable branch of $q$ and forms a path of heteroclinic connections from $p$ to $q$. 

Therefore $L$ always serve as a saddle connection
under the assumption that  $L$ contains no recurrent point. This finishes the proof.
\end{proof}

\begin{cor}[\cite{Mat2}]\label{closure}
Suppose that each periodic point of $f$ is either hyperbolic, or elliptic and Moser stable.
If $f$ doesn't have any path of saddle connection, then
all four branches of any hyperbolic periodic point are recurrent and have the same closure.
\end{cor}
\begin{proof}
Let $L$, $\Gamma$, $L^-$ and $\Gamma^-$ be the four branches of a  hyperbolic periodic point $p$ 
(ordered counterclockwise).
It suffices to show that the closure of any given branch contains the branch next to it. 
Then by induction we see that the closure contains all four branches and serves as the common closure.
Without loss of generality, we assume $L$ is an unstable branch and show $\Gamma\subset \overline{L}$.

If this is not the case, then $\Gamma\cap \overline{L}=\emptyset$ (by Lemma \ref{branch}).
Let $U$ be a connected component of $M\backslash \overline{L}$ containing $\Gamma$. 
Clearly $p$ is an accessible boundary point from $U$ (along $\Gamma$). 
Then by (the proof of) Lemma \ref{open}, there exists an open neighborhood 
$W\ni p$ such that $\partial U\cap W$ consists of one or two local branches. 
Clearly it can't be the single branch $L$ (by the recurrence from Lemma \ref{recurrent}).
It can't contain $\Gamma$ (since $\Gamma\cap \overline{L}=\emptyset$), 
or $L^-$ (both $L$ and $L^-$ are unstable branches). So it has to contain $\Gamma^-$,
which leads to a homoclinic loop $L=\Gamma^-$ and contradicts the assumption
that there is no path of saddle connection. This completes the proof.
\end{proof}

\section{Homoclinic points on generic convex billiard systems}

Let $\cU_n$ be the set of convex billiard tables given in \S \ref{Un}, and 
$\cR=\bigcap_{n\ge1}\cU_n$. Then by Lemma \ref{finite}, 
we see that $\cR$ contains a $C^\infty$-residual subset of 
$\cQ^\infty$, such that for every billiard table  $Q\in \cR$, 
for the billiard map $F$ on the phase space $M=\partial Q\times(-1,1)$,
\begin{enumerate}
\item[(a)] each periodic point of $F$ is either hyperbolic or elliptic;

\item[(b)] given two hyperbolic periodic points $p,q$, 
$W^u(p)$ and $W^s(q)$ either don't intersect, or have some transversal intersections;

\item[(c)]  every elliptic periodic point  is Moser stable.
\end{enumerate}

In the following, let $p$ be a hyperbolic periodic point of $F$,
$L$ be a branch of the unstable manifold of $p$,
$K=\overline{L}$, which is compact, connected, $f$-invariant,
where $f=F^n$ (or $F^{2n}$ if necessary).
According to Proposition \ref{finite}, we can assume 
$p\in\cA_n=\{(r,s):|s|\le 1-1/n\}$ without loss of generality.
According to Lazutkin Theorem, there always exist invariant circles 
accumulating to both boundaries $\{(r,s):s=\pm 1\}$ of $M$.
So what happens near the boundary does not affect dynamics in the middle.
We can collapse $\partial M$ to two points, say $N$ and $S$, and the result surface is homeomorphic
to $S^2$. Clearly both $N$ and $S$ are surrounded by invariant closed disks 
enclosed by Lazutkin invariant circles and have Moser-stable behaviors.
The following is a direct application of Proposition \ref{recurrent}, 
since the paths of saddle connections are excluded from systems in $\cR$.
\begin{cor}\label{billrecu}
Let $Q\in\cR$ and $p$ be a hyperbolic periodic point.
Then the two unstable branches are recurrent under $f$, 
and two stable branches are recurrent under $f^{-1}$.
\end{cor}

\begin{pro}\label{accumu}
There always exists an adjacent pair $(K,L)$ of stable and unstable branches such
that they accumulate on each other from the quadrant between them.
\end{pro}
\begin{proof}
Let's start with the branch $L$ of unstable manifold from the positive $x$-axis. 
Let $q$ be a recurrent point given by Corollary \ref{billrecu}. 
So some iterates $f^{n_k}q$ accumulates on $q$ itself, from either the first quadrant $Q1$
or from the fourth quadrant $Q4$ (possibly both). We may assume that it is the
former case, the orbit of q accumulates from $Q1$. 
Then it is easy to see the branch $K$ of the stable manifold from the positive $y$-axis
is contained in $\omega(L)$.

There is a backward recurrent point $q$ on $K$, whose backward orbit 
accumulates on itself either from $Q1$ or
the second quadrant $Q2$. In the former case, we have an adjacent pair of stable
and unstable branches $(L,K)$ that accumulate on each other. 

Now let's assume the backward orbit of $q$
accumulates on itself from $Q2$, which implies $L^-\subset \alpha(K)$.
Now let's consider $L^-$, which can accumulate to itself from either $Q2$,
or from the third quadrant $Q3$. If it accumulates from $Q2$, 
then we have an adjacent pair $(K,L^-)$. 
Otherwise let's check $K^-$. Either we have an adjacent pair $(L^-,K^-)$, 
or $K^-$ accumulates on $L$ though $Q4$
and the accumulating relations of four branches form a closed chain.
However, the accumulating relations are transitive.
Since $K^-$ accumulates to itself though $Q4$, $L$ also accumulates on $K^-$ though $Q4$.
This still gives rise to an adjacent accumulating pair $(L,K^-)$. 
\end{proof}

\subsection{Recurrence and homoclinic orbits}

\begin{thm}
Let $Q\in\cR$ and $p$ be a hyperbolic periodic point with an adjacent accumulating pair.
Then there exists a homoclinic orbit of $p$. 
\end{thm}
\begin{proof}
Let $(L,K)$ be the adjacent accumulating pair given by Proposition \ref{accumu}.
Without loss of generality, we assume
$L$ is the branch of the unstable manifold of $p$ along the positive $x$-axis,
$K$ is the branch of the stable manifold of $p$ along the positive $y$-axis,
and they accumulate each other from the first quadrant.

Let $W$ be a small $C^0$ chart around $p$ given by Hartman's theorem,
$\delta$ be a small positive number, $m\ge 4$ be a large integer, and
$S_{\delta,m}=\{(x,y):0\le x,y\le \delta,xy\le \lambda^{-m}\delta\}$ be a small region in
the first quadrant of $W$.
Let $q_u$ be the position on $L$ that hits $S_{\delta,m}$ for first time,
and $q_s$ be the position on $K$ that hits $S_{\delta,m}$ for first time.

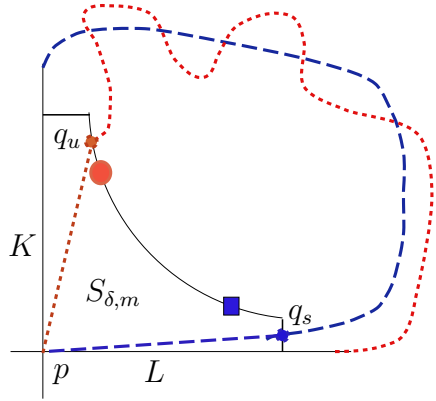
\begin{figure}[h]
\begin{overpic}[width=60mm]{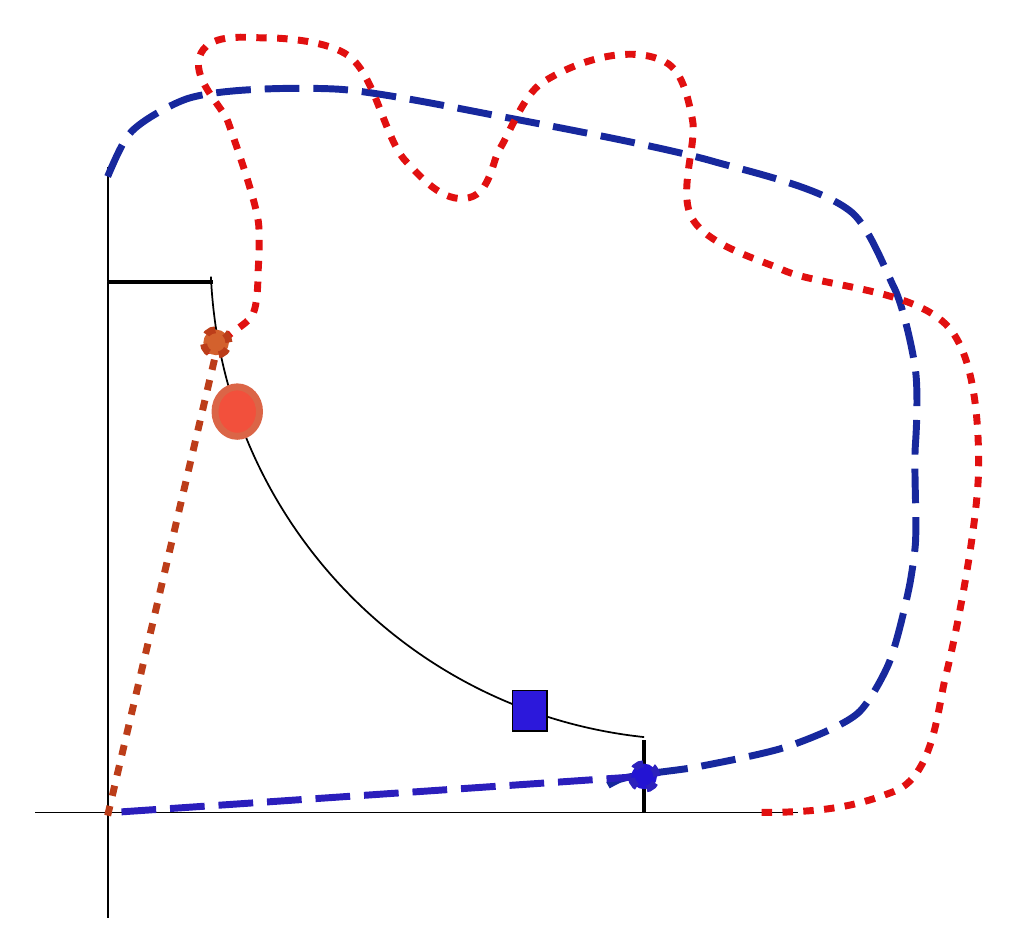}
\put(13,8){$p$}
\put(33,7){$L$}
\put(3,35){$K$}
\put(20,25){$S_{\delta,m}$}
\put(13,60){$q_u$}
\put(65,21){$q_s$}
\end{overpic}
\caption{Existence of homoclinic points from adjacent pair of branches.}
\label{adjacent}
\end{figure}

According  our choice of coordinate system,
$q_u$  lies on the part of boundary $\partial S_{\delta,m}$ above the round spot, which is given by 
$$\{(x,\delta):0\le x\le\lambda^{-m}\}\cup \{(x,y):
xy=\lambda^{-m}\delta, \lambda^{-m}\le x\le\lambda^{1-m}\}.$$
Similarly, $q_s$  lies on  the part of boundary $\partial S_{\delta,m}$ below the square spot, which is given by
$$\{(\delta,y):0\le y\le\lambda^{-m}\}\cup \{(x,y):
xy=\lambda^{-m}\delta, \lambda^{-m}\le y\le\lambda^{1-m}\}.$$

We already have a homoclinic intersection if either $q_u\in K$ or $q_s\in L$.
In the following we assume this is not the case (see Figure \ref{adjacent}). 
Let's consider two simple closed curves $C_u=L[p,q_u]\ast [q_u,p]$
and $C_s=K[p,q_s]\ast [q_s,p]$.
Note that on $S^2$, the algebraic sum 
of the intersection numbers of two closed curves must be zero.
Clearly $C_u$ and $C_s$ cross each other once at $p$.
So  $C_u$ and $C_s$ have to cross each other elsewhere, say $q$.
Note that $q_u$ and $q_s$ are separated by the two shaded spots (since $m\ge 4$).
So $q$ lies outside $S_{\delta,m}$, which forces $q\in L(p,q_u)\cap K(p,q_s)\neq\emptyset$.
That is, $q$ is a homoclinic point of $p$. 
\end{proof}

\begin{cor}\label{4branch}
Let $Q\in\cR$ and $p$ be a hyperbolic periodic point.
Then there exist some transversal homoclinic orbits on all branches of $p$.
\end{cor}
\begin{proof}
As in the previous proof, let $q$ be a homoclinic intersection of some pair $(L,K)$ of the branches.
Without loss of generality, we assume they intersect transversally at $q$ 
(by the choice of the table $Q\in\cR$).
Then  $L[p,q]\ast K[q,p]$ is a simple closed curve. Let $D$ be the domain enclosed by it.
It is easy to see the other branches, $L^-$ and $K^-$, 
have to enter $D$ (by Corollary \ref{closure}), since some parts of $K$ and $L$ are contained in $D$.
So $L^-$ has to cross $K$, and $K^-$ has to cross $L$. 
The crossing points provide the homoclinic points on these branches.
Again our choice of $\cR$ implies the existence of transversal homoclinic orbits.
This finishes the proof.
\end{proof}

\section*{Acknowledgments}
This research is supported in part by National Science Foundation.

\end{document}